\documentclass[journal]{IEEEtran}
\usepackage{amsmath}    
\usepackage{graphics,epsfig,subfigure}    
\usepackage{verbatim}   
\usepackage{color}      
\usepackage{subfigure}  
\usepackage{hyperref}   
\usepackage{amssymb}
\usepackage{epstopdf}
\usepackage{graphicx}
\newcommand{\expect}[1]{\mathbb{E}\big\{#1\big\}}

\newcommand{\bv}[1]{{\boldsymbol{#1} }}
\newcommand{\script}[1]{{{\cal{#1} }}}

\begin{document}
\title{Max-Weight Achieves the Exact $[O(1/V), O(V)]$ Utility-Delay Tradeoff Under Markov Dynamics}
\author{\large{Longbo Huang, Michael J. Neely}%
\thanks{Longbo Huang (web:  http://www-scf.usc.edu/$\sim$longbohu)
and Michael J. Neely (web:  http://www-rcf.usc.edu/$\sim$mjneely)
are with the Department of Electrical
Engineering, University of Southern California, Los Angeles, CA 90089, USA.}%
\thanks{This material is supported in part  by one or more of 
the following: the DARPA IT-MANET program
grant W911NF-07-0028, the NSF grant OCE 0520324, 
the NSF Career grant CCF-0747525.} }
\maketitle

\newtheorem{rem}{Remark}
\newtheorem{fact_def}{\textbf{Fact}}
\newtheorem{coro}{\textbf{Corollary}}
\newtheorem{lemma}{\textbf{Lemma}}
\newtheorem{main}{\textbf{Proposition}}
\newtheorem{thm}{\textbf{Theorem}}
\newtheorem{claim}{\emph{Claim}}
\newtheorem{prop}{Proposition}
\newtheorem{assumption}{\textbf{Assumption}}

 \begin{abstract}
In this paper, we show that the Quadratic Lyapunov function based Algorithm (QLA, also  known as MaxWeight or Backpressure) achieves an exact $[O(1/V), O(V)]$ utility-delay tradeoff in stochastic network optimization problems with Markovian network dynamics. Note that though the QLA algorithm has been extensively studied, most of the performance results are obtained under i.i.d. network radnomness, and it has not been formally proven that QLA achieves the exact $[O(1/V), O(V)]$ utility-delay tradeoff under Markov dynamics. 
Our analysis uses a combination of duality theory and a variable multi-slot Lyapunov drift argument. The variable multi-slot Lapunov  drift argument here is different from previous multi-slot drift analysis, in that the slot number is a random variable corresponding to the renewal time of the network randomness. This variable multi-slot drift argument not only allows us to obtain an exact $[O(1/V), O(V)]$ tradeoff, but also allows us to state the performance of QLA in terms of explicit parameters of the network dynamic process. 
 \end{abstract}

\begin{keywords}
Queueing, Dynamic Control, Lyapunov analysis,  Stochastic Optimization
\end{keywords}

\section{Introduction}
In this paper, we show that the Quadratic Lyapunov function based Algorithm (QLA, also known as the MaxWeight algorithm) \cite{neelynowbook} achieves an exact $[O(1/V), O(V)]$ utility-delay tradeoff in the following general stochastic network optimization problem.  We are given a discrete time stochastic network. The network state, which describes  the network randomness, such as the network channel condition or the random arrivals, is time varying according to some Markov process. A network controller performs some action based on the observed network state at every time slot. The chosen action incurs a cost, \footnote{Since cost minimization is mathematically equivalent to utility maximization, below we will use cost and utility interchangeably} but also serves some amount of traffic and possibly generates new traffic for the network.  This traffic causes congestion, and thus leads to backlogs at nodes in the network. The goal of the controller is to minimize its time average cost subject to the constraint that the time average total backlog in the network is finite. 

This is a very general framework and includes a wide class of networking problems, ranging from flow utility maximization \cite{eryilmaz_qbsc_ton07}, energy minimization \cite{neelyenergy}, network pricing \cite{huangneelypricing} to cognitive radio applications \cite{rahulneelycognitive} etc. Also,  many techniques have also been applied to this problem (see \cite{yichiang_netopt08} for a survey).  
Among the many techniques that have been adopted, the family of Quadratic Lyapunov function based Algorithms (QLA) \cite{neelynowbook} are recently receiving much attention, due to their provable performance guarantees, robustness to stochastic network conditions, and most importantly, their ability to achieve the desired performance \emph{without requiring any statistical knowledge} of the underlying randomness in the network. When the network state is i.i.d., it has been proven in \cite{neelynowbook} that QLA can achieve a utility that is within $O(1/V)$ of  the optimal utility for any $V\geq1$ for general network optimization problems, while guaranteeing an $O(V)$ network delay. Two works \cite{neelysuperfast} \cite{neelyenergydelay} construct algorithms to achieve an $\big[O(1/V), O(\log(V))\big]$ utility-delay tradeoff using exponential Lyapunov functions. The recent work \cite{huangneely_dr_wiopt09} also develops the Fast-QLA  (FQLA) algorithm based on quadratic Lyapunov functions to achieve an $\big[O(1/V), O([\log(V)]^2)\big]$ tradeoff. 

When the network state is Markovian, it has been shown that when the network backlogs are deterministically bounded, QLA can also achieve utilities within $O(\log(V)/V)$ to the optimal values \cite{huangneelypricing} \cite{rahulneelycognitive}, while guaranteeing that the average delay is $O(V)$. Without such deterministic queueing bounds, it has recently been shown that QLA achieves an $[O(\epsilon+\frac{T_{\epsilon}}{V}), O(V)]$ tradeoff under Markovian network states \cite{neely_queuestability10}, where $\epsilon>0$ and $T_{\epsilon}$ represent the proximity to the optimal value and the ``convergence time'' of the QLA algorithm to that proximity, respectively. However, there has not been any proof showing that QLA achieves the exact $[O(1/V), O(V)]$ utility-delay tradeoff under Markovian network dynamics. 


In this paper, we present the first proof of the exact $[O(1/V), O(V)]$ tradeoff of the QLA algorithm under Markovian network dynamics. To establish the result, we use a combination of duality theory and a variable multi-slot Lyapunov drift argument. Different from previous multi-slot drift arguements, e.g.,\cite{neelynowbook}, where the drift is usually computed over a fixed number of slots, the slot number here is a random variable corresponding to the renewal time of the network dynamic process. 
This $[O(1/V), O(V)]$ tradeoff result contributes to a better understanding of the QLA algorithm performance and enables more precise resource allocation  decisions in network optimization problems. The result can also be combined with the recent result developed in \cite{huangneely_dr_wiopt09} to show that the FQLA algorithm achieves an $\big[O(1/V), O([\log(V)]^2)\big]$ tradeoff for stochastic network optimization problems with Markovian network dynamics, and is thus the  first known algorithm that can ensure a poly-logarithmic delay performance when pushing the utility performance to within $O(1/V)$ of the optimal in this Markovian case. 

This paper is organized as follows. In Section \ref{section:notation}, we set up our  notations. We then present our system model in Section \ref{section:model}. We review the QLA algorithm in Section \ref{section:qlareview}. The performance results of QLA under the Markovian network dynamics are obtained in Section \ref{section:qla-performance}. 


\section{Notations}\label{section:notation}
Here we specify our notations. $\mathbb{R}$ represents the set of real numbers. $\mathbb{R}_+$ (or $\mathbb{R}_-$) represents the set of nonnegative (or non-positive) real numbers. $\mathbb{R}^n$ (or $\mathbb{R}^n_+$) represents the set of $n$ dimensional \emph{column} vectors, with each element being in $\mathbb{R}$ (or $\mathbb{R}_+$).  \textbf{bold} symbols $\bv{a}$ and $\bv{a}^T$ represent \emph{column} vector and its transpose. $\bv{a}\succeq\bv{b}$ indicates that vector $\bv{a}$ is entrywise no less than vector $\bv{b}$. $||\bv{a}-\bv{b}||$ is the Euclidean distance of $\bv{a}$ and $\bv{b}$. $\bv{0}$ is the column vector with all elements being $0$.

$\vspace{-.2in}$
\section{System Model}\label{section:model}
In this section, we specify the general network model we use. We consider a network controller that operates a network with the goal of minimizing the time average cost, subject to the queue stability constraint. The network is assumed to operate in slotted time, i.e., $t\in\{0,1,2,...\}$. We assume there are $r\geq1$ queues in the network. 

$\vspace{-.22in}$
\subsection{Network State}
In every slot $t$, we use $S(t)$ to denote the current network state, which indicates the current network parameters, such as a vector of channel conditions for each link, or a collection of other relevant information about the current network channels and arrivals. 
We assume that $S(t)$ evolves according to a general irreducible and aperiodic Markov chain with countably many states and denote its state space by $\script{S} = \{s_1, s_2, s_3, \ldots\}$. We assume $S(t)$ has a well defined steady state distribution, and let $\pi_{s_i}$ denote its steady state probability of being in state $s_i$. Note that in this case, by Theorem 3 in Chapter 5 of \cite{gallager_stochastic}, the existence of a steady state distribution $\bv{\pi}$ implies that all the states are positive recurrent, hence $\pi_{s_i}>0$ for all $i$. 




$\vspace{-.22in}$
\subsection{The Cost, Traffic, and Service}\label{subsection:costtrafficservice}
At each time $t$, after observing $S(t)=s_i$, the controller chooses an action $x(t)$ from a set $\script{X}^{(s_i)}$, i.e., $x(t)= x^{(s_i)}$ for some $x^{(s_i)}\in\script{X}^{(s_i)}$. The set $\script{X}^{(s_i)}$ is called the feasible action set for network state $s_i$ and is assumed to be time-invariant and compact for all $s_i\in\script{S}$.  The cost, traffic, and service generated by the chosen action $x(t)=x^{(s_i)}$ are as follows:
\begin{enumerate}
\item[(a)] The chosen action has an associated cost given by the cost function $f(t)=f(s_i, x^{(s_i)}): \script{X}^{(s_i)}\mapsto \mathbb{R}_+$ (or $\script{X}^{(s_i)}\mapsto\mathbb{R}_-$ in reward maximization problems);

\item[(b)] The amount of traffic generated by the action to queue $j$ is determined by the traffic function $A_j(t)=A_{j}(s_i, x^{(s_i)}): \script{X}^{(s_i)}\mapsto \mathbb{R}_{+}$, in units of packets; 

\item[(c)] The amount of service allocated to queue $j$ is given by the rate function $\mu_j(t)=\mu_{j}(s_i, x^{(s_i)}): \script{X}^{(s_i)}\mapsto \mathbb{R}_{+}$, in units of packets;

 \end{enumerate}
Note that $A_j(t)$ includes both the exogenous arrivals from outside the network to queue $j$, and the endogenous arrivals from other queues, i.e., the transmitted packets from other queues, to queue $j$. We assume the functions $f(s_i, \cdot)$, $\mu_{j}(s_i, \cdot)$ and $A_{j}(s_i, \cdot)$ are time-invariant, their magnitudes are uniformly upper bounded by some constant $\delta_{max}\in(0,\infty)$ for all $s_i$, $j$, and they are known to the network operator. We also assume that there exists a set of actions $\{x^{(s_i)}_k\}_{i=1, 2, ...}^{k=1,2, ..., r+2}$ with $x^{(s_i)}_k\in\script{X}^{(s_i)}$ for all $s_i$, and a set of variables $\{\vartheta^{(s_i)}_k\}_{i=1, 2, ...}^{k=1,2, ..., r+2}$ with $\sum_k\vartheta^{(s_i)}_k=1$  and $\vartheta^{(s_i)}_k\geq0$ for all $s_i$ and $k$ such that:  
\begin{eqnarray}
\sum_{s_i}\pi_{s_i}\big\{\sum_k\vartheta^{(s_i)}_k[A_{j}(s_i, x^{(s_i)}_k)-\mu_{j}(s_i, x^{(s_i)}_k)]\big\}\leq -\eta,\label{eq:slackness}
\end{eqnarray}
for some $\eta>0$ for all $j$. That is, the queue stability 
constraints are feasible with $\eta$-slackness. Thus, there exists a stationary randomized policy that stabilizes all queues (where $\vartheta^{(s_i)}_k$ represents the probability of choosing action $x^{(s_i)}_k$ when $S(t)=s_i$).  In the following, we use $\bv{A}(t)=(A_1(t), A_2(t), ..., A_r(t))^{T}$ and $\bv{\mu}(t)=(\mu_1(t), \mu_2(t), ..., \mu_r(t))^{T}$ to denote the arrival and service vectors at time $t$. It is easy to see from above that if we define:
\begin{eqnarray}
B=\sqrt{r}\delta_{max},\label{eq:Bdef}
\end{eqnarray}
then $\|\bv{A}(t)-\bv{\mu}(t)\|\leq B$ for all $t$. 

$\vspace{-.25in}$
\subsection{Queueing, Average Cost, and the Stochastic Problem}\label{section:queuenotation}
Let $\bv{q}(t)=(q_1(t), ..., q_r(t))^T\in\mathbb{R}^r_{+}$, $t=0, 1, 2, ...$ be the queue backlog vector  process of the network, in units of packets. We assume the following queueing dynamics: 
\begin{eqnarray}
q_j(t+1)=\max\big[q_j(t)-\mu_j(t), 0\big]+A_j(t)\quad\forall j,\label{eq:queuedynamic}
\end{eqnarray}
and $\bv{q}(0)=\bv{0}$. By using (\ref{eq:queuedynamic}), we assume that when a queue does not have enough packets to send, null packets are transmitted.  In this paper, we adopt the following notion of queue stability:
\begin{eqnarray}
\expect{\sum_{j=1}^rq_j}\triangleq
\limsup_{t\rightarrow\infty}\frac{1}{t}\sum_{\tau=0}^{t-1}\sum_{j=1}^{r}\expect{q_j(\tau)}<\infty.\label{eq:queuestable}
\end{eqnarray}
We also use $f^{\Pi}_{av}$ to denote the time average cost induced by an action-choosing policy $\Pi$, defined as:
\begin{eqnarray}
f^{\Pi}_{av}\triangleq
\limsup_{t\rightarrow\infty}\frac{1}{t}\sum_{\tau=0}^{t-1}\expect{f^{\Pi}(\tau)},\label{eq:timeavcost}
\end{eqnarray}
where $f^{\Pi}_{av}(\tau)$ is the cost incurred at time $\tau$ by policy $\Pi$. We call an action-choosing  policy \emph{feasible} if at every time slot $t$, it only chooses actions from the feasible action set $\script{X}^{(S(t))}$.  We then call a feasible action-choosing  policy under which (\ref{eq:queuestable}) holds a \emph{stable} policy, and use $f_{av}^*$ to denote the optimal time average cost over all stable policies. 
In every slot, the network controller observes the current network state $S(t)$ and chooses a control action, with the goal of minimizing time average cost subject to network stability. This goal can be mathematically stated as:
\textbf{(P1)}\,\,\, $\bv{\min: \,  f^{\pi}_{av}, \,\, s.t.\,  (\ref{eq:queuestable})}$. 
In the rest of the paper, we will refer to problem (\textbf{P1}) as \emph{the stochastic problem}. 

\section{QLA and the Deterministic Problem}\label{section:qlareview}
In this section, we first review the quadratic Lyapunov function based algorithms (the QLA algorithm) \cite{neelynowbook} for solving the stochastic problem. Then we define the \emph{deterministic problem} and its dual problem. We then also discuss some properties of the dual function. The dual problem and the properties of the dual function will be used later for analyzing the performance of QLA. 


We first recall the QLA algorithm \cite{neelynowbook} as follows. 

\underline{\emph{QLA:}} Initialize the parameter $V\geq1$. At every time slot $t$, observe the current network state $S(t)$ and the backlog $\bv{q}(t)$. If $S(t)=s_i$, choose $x^{(s_i)}\in\script{X}^{(s_i)}$ that solves the following: 
\begin{eqnarray}
\hspace{-.3in}\max: && -Vf(s_i, x)+\sum_{j=1}^{r}q_j(t)\big[\mu_j(s_i, x)-A_j(s_i, x)\big]\label{eq:QLAeq}\\
s.t. && x\in\script{X}^{(s_i)}.\nonumber
\end{eqnarray}
Depending on the problem structure, (\ref{eq:QLAeq}) can usually be decomposed into separate parts that are easier to solve, e.g., \cite{neelyenergy}, \cite{huangneelypricing}. 
Also, when the network state process $S(t)$ is i.i.d., it has been shown in \cite{neelynowbook} that, 
\begin{eqnarray}
f_{av}^{QLA}=f^*_{av}+O(1/V),\quad \overline{q}^{QLA}=O(V),\label{eq:qla_performance}
\end{eqnarray}
where $f_{av}^{QLA}$ and $\overline{q}^{QLA}$ are the expected average cost and the expected time average network backlog size under QLA, respectively. 
When $S(t)$ is Markovian, it has been shown in, e.g., \cite{huangneelypricing} and \cite{rahulneelycognitive} that QLA achieves an $[O(\log(V)/V), O(V)]$ utility-delay tradeoff if the queue sizes are deterministically upper bounded by $\Theta(V)$ for all time. Without this deterministic backlog bound, it has recently been shown that QLA achieves an $[O(\epsilon+\frac{T_{\epsilon}}{V}), O(V)]$ tradeoff under Markov $S(t)$ processes, where $\epsilon$ and $T_{\epsilon}$ represent the proximity to the optimal value and the ``convergence time'' of the QLA algorithm for this proximity \cite{neely_queuestability10}. 
However, this latter tradeoff is less explicit, and it is common that when $S(t)$ is Markovian, $T_{\epsilon}=\Omega(\log(\frac{1}{\epsilon}))$, in which case we again have an $[O(\frac{\log(V)}{V}), O(V)]$ tradeoff when $\epsilon=1/V$. 

We also recall the \emph{the deterministic problem} defined in \cite{huangneely_dr_wiopt09}:
\begin{eqnarray}
\min:&&\script{F}(\bv{x})\triangleq V\sum_{s_i}\pi_{s_i}f(s_i, x^{(s_i)})\label{eq:primal}\\
s.t.&&\script{A}_j(\bv{x})\triangleq\sum_{s_i}\pi_{s_i}A_j(s_i, x^{(s_i)})\nonumber\\
&&\qquad\qquad\quad\leq \script{B}_j(\bv{x})\triangleq\sum_{s_i}\pi_{s_i}\mu_j(s_i, x^{(s_i)})\quad\forall\, j\nonumber\\
&& x^{(s_i)}\in \script{X}^{(s_i)}\quad \forall\, i=1, 2, ...  \nonumber
\end{eqnarray}
where $\pi_{s_i}$ corresponds to the steady state probability of $S(t)=s_i$ and $\bv{x}=(x^{(s_1)}, ..., x^{(s_M)})^T$. The dual problem of (\ref{eq:primal}) can be obtained as follows:
\begin{eqnarray}
\max:\,\,\, g(\bv{\gamma}),\quad s.t.\,\,\, \bv{\gamma}\succeq\bv{0},\label{eq:dualproblem}
\end{eqnarray}
where $g(\bv{\gamma})$ is called the dual function and is defined as:
\begin{eqnarray}
g(\bv{\gamma})=\inf_{x^{(s_i)}\in \script{X}^{(s_i)}}\sum_{s_i}\pi_{s_i}\bigg\{Vf(s_i, x^{(s_i)})\label{eq:dual_separable}\qquad\qquad\qquad\\
+\sum_j\gamma_j\big[A_j(s_i, x^{(s_i)})- \mu_j(s_i, x^{(s_i)})\big]\bigg\}.\nonumber
\end{eqnarray}

Here $\bv{\gamma}=(\gamma_1, ..., \gamma_r)^T$ is the \emph{Lagrange multiplier} of (\ref{eq:primal}). It is well known that $g(\bv{\gamma})$ in (\ref{eq:dual_separable}) is concave in the vector $\bv{\gamma}$, and hence the problem (\ref{eq:dualproblem}) can usually be solved efficiently, particularly when cost functions and rate functions are separable over different network components. It is also well known that in many situations, the optimal value of (\ref{eq:dualproblem}) is the same as the optimal value of (\ref{eq:primal}) and in this case we say that there is no duality gap \cite{bertsekasoptbook}. However, despite the fact that the problem (\ref{eq:primal}) may be non-convex, in which case the duality gap is usually nonzero, our first result shows that the dual problem (\ref{eq:dualproblem}) gives the exact  value of $Vf_{av}^*$, where $f_{av}^*$ is the optimal time average cost for the stochastic problem. Below, $\bv{\gamma}_V^*=(\gamma^*_{V1}, \gamma^*_{V2}, ..., \gamma^*_{Vr})^T$ denotes an optimal solution of the dual problem (\ref{eq:dualproblem}) with the corresponding $V$ parameter. 

\begin{thm}\label{thm:primaldualequal}
Let $\bv{\gamma}_V^*$ be an optimal solution of the dual problem (\ref{eq:dualproblem}). We  have:
\begin{eqnarray}
g(\bv{\gamma}^*_V)=Vf^*_{av}.
\end{eqnarray}
\end{thm}
\begin{proof}
See Appendix A. 
\end{proof}

The following corollary is immediate and will be useful for our following analysis.
\begin{coro}\label{coro:primaldualequal}
For any $\bv{\gamma}\succeq\bv{0}$, we have: 
\begin{eqnarray}
g(\bv{\gamma})\leq Vf^*_{av}.
\end{eqnarray}
\end{coro}


In the following, we also define the functions $g_{s_i}(\bv{\gamma})$ for each $s_i=s_1, s_2, ...$ as follows:
\begin{eqnarray}
g_{s_i}(\bv{\gamma})=\inf_{x^{(s_i)}\in \script{X}^{(s_i)}} \bigg\{Vf(s_i, x^{(s_i)})\label{eq:dual_separable_si}\qquad\qquad\qquad\qquad\\
+\sum_j\gamma_j\big[A_j(s_i, x^{(s_i)})- \mu_j(s_i, x^{(s_i)})\big]\bigg\}.\nonumber
\end{eqnarray}
That is, the $g_{s_i}(\cdot)$ function is the dual function of (\ref{eq:primal}) when the network has only one single network state $s_i$, i.e., the network condition is deterministically described by $s_i$. It is easy to see from (\ref{eq:dual_separable}) and (\ref{eq:dual_separable_si}) that:
\begin{eqnarray}
g(\bv{\gamma})=\sum_{s_i}\pi_{s_i}g_{s_i}(\bv{\gamma}). \label{eq:dualfunction_sumofsi}
\end{eqnarray} 
Also, the term $\bv{G}^{(s_i)}_{\bv{\gamma}}=(G^{(s_i)}_{\bv{\gamma},1}, G^{(s_i)}_{\bv{\gamma},2}, ..., G^{(s_i)}_{\bv{\gamma},r})^{T}$ with: 
\begin{eqnarray}
G^{(s_i)}_{\bv{\gamma},j}
=\big[-\mu_j(s_i, x_{\bv{\gamma}}^{(s_i)})+ A_j(s_i, x_{\bv{\gamma}}^{(s_i)})\big],\label{eq:subgradient_def_si}
\end{eqnarray}
is called the subgradient of the $g_{s_i}(\cdot)$ function at the point $\bv{\gamma}$ \cite{bertsekasoptbook}.  It is known that for any other $\hat{\bv{\gamma}}\in\mathbb{R}^r$, we have:
\begin{eqnarray}
(\hat{\bv{\gamma}}-\bv{\gamma})^{T}\bv{G}^{(s_i)}_{\bv{\gamma}}\geq g_{s_i}(\hat{\bv{\gamma}})-g_{s_i}(\bv{\gamma}).\label{eq:subgradientpro}
\end{eqnarray}
Using the fact that $\|\bv{G}^{(s_i)}_{\bv{\gamma}}\|\leq B$, (\ref{eq:subgradientpro}) also implies:
\begin{eqnarray}
g_{s_i}(\hat{\bv{\gamma}})-g_{s_i}(\bv{\gamma})\leq B\|\hat{\bv{\gamma}}-\bv{\gamma}\|.\label{eq:dualbddslope_si}
\end{eqnarray}

\section{Performance of QLA under Markovian Dynamics}\label{section:qla-performance}
In this section, we prove that under the Markovian network state dynamics, QLA achieves an exact $[O(\frac{1}{V}), O(V)]$ utility-delay tradeoff for the stochastic problem. This is the first formal proof of this result. It generalizes the $[O(\frac{1}{V}), O(V)]$ performance result of QLA in the i.i.d. case in \cite{neelynowbook}. 
To prove the result, we use a variable multi-slot Lyapunov drift argument. Different from previous multi-slot drift arguments, e.g., \cite{neelypowerjsac} and \cite{neely_queuestability10}, where the drift is usually computed over a fixed number of slots, the slot number here is a random variable corresponding to the return time of the network states. As we will see, this variable multi-slot drift analysis allows us to obtain the exact $[O(\frac{1}{V}), O(V)]$ utility-delay tradeoff for QLA. Moreover, it also allows us to state QLA's performance in terms of explicit parameters of the Markovian $S(t)$ process. 

In the following, we define $T_{i}(t_0)$ to be the first return time of $S(t)$ to state $s_i$ given that $S(t_0)=s_i$, i.e., 
\begin{eqnarray*}
T_{i}(t_0)=\inf\{T>0, \,s.t.\, S(t_0+T)=s_i\, |\, S(t_0)=s_i\}.
\end{eqnarray*} 
We see that $T_{i}(t_0)$ has the same distribution for all $t_0$. Thus, we will use $\overline{T_{i}}$ to denote the expected value of $T_{i}(t)$ for any $t$ s.t. $S(t)=s_i$ and use $\overline{T^2_{i}}$ to denote its second moment. 
By Theorem 3 in Chapter 5 of \cite{gallager_stochastic}, we have for all states $s_i$ that:
\begin{eqnarray}
\overline{T_{i}} =\frac{1}{\pi_{s_i}}<\infty, \label{eq:exp-time-prob}
\end{eqnarray}
i.e., the expected return time of any state $s_i$ is finite. 
In the following, we also use $T_{ji}(t_0)$ to denote the first hitting time for $S(t)$ to enter the state $s_i$ given that $S(t_0)=s_j$. It is again easy to see that $T_{ji}(t_0)$ has the same distribution at all $t_0$. Hence we similarly use $\overline{T_{ji}}$ and $\overline{T_{ji}^2}$ to denote its first and second moments. 
Throughout the paper, we make the following assumption:
\begin{assumption}\label{assumption:finitereturn}
There exists a state $s_1$ such that: \[\overline{T_{j1}^2}<\infty,\quad\forall\,\, j.\] 
\end{assumption}
That is, starting from any state $s_j$ (including $s_1$), the random time needed to get into state $s_1$ has a finite second moment. 
This condition is not very restrictive and can be satisfied in many cases, e.g., when $\script{S}$ is finite. 


We now have the following theorem summarizing QLA's performance under the Markovian network state dynamics:
\begin{thm}\label{thm:generalq}
Suppose (\ref{eq:slackness}) holds. Then 
under the Markovian network state process $S(t)$, the QLA algorithm achieves the following:
\begin{eqnarray}
\hspace{-.3in}f^{QLA}_{av}&\leq& f^*_{av}+\frac{CB^2}{V\overline{T_1}},\label{eq:qla_utility}\\
\hspace{-.3in}\overline{\sum_{j=1}^rq_j} &\leq& \frac{CB^2 + \overline{T_1}V\delta_{max}}{ \eta} +\frac{DB^2}{2}, \label{eq:qla_backlog}
\end{eqnarray}
where $\eta>0$ is the slack parameter defined in (\ref{eq:slackness}) in Section \ref{subsection:costtrafficservice}, and $C, D$ are defined as:  
\begin{eqnarray}
C=\overline{T^2_1}+\overline{T_1},\,\,
D=\overline{T^2_1}-\overline{T_1},\label{eq:max_returnmoment_def}
\end{eqnarray}
i.e., $C$ and $D$ are the sum and difference of the first and second moments of the return time associated with $s_1$.
\end{thm}

Note that $\eta, C, D=\Theta(1)$ in (\ref{eq:qla_backlog}), i.e., independent of $V$. Hence Theorem \ref{thm:generalq} shows that QLA indeed achieves an exact $[O(1/V), O(V)]$ utility-delay tradeoff for general stochastic network optimization problems with  Markovian network dynamics. Although our bounds  may be loose when the number of states is large, we note that Theorem \ref{thm:generalq} also applies to the case when $S(t)$ evolves according to a Markov modulated i.i.d. process, in which case there is a Markov chain of only a few states, but in each Markov state, there can be many i.i.d. randomness. For example, suppose $S(t)$ is i.i.d. with $10^4$ states. Then we can view $S(t)$ as having one Markov state, but within the Markov state, it has $10^4$ i.i.d. random choices. In this case, Theorem \ref{thm:generalq} will apply with $C=2$ and $D=0$. 
These Markov modulated processes can easily be incorporated into our analysis by taking expectation over the i.i.d. randomness of the current Markov state in Equation (\ref{eq:driftcompute_noexp}). 
These Markov modulated processes are important in stochastic modeling and include the $ON/OFF$ processes for modeling time-correlated arrivals processes, e.g., \cite{neely_maximaldelay_ton}. 

\begin{proof} (Theorem \ref{thm:generalq}) 
To prove the theorem, we first define the Lyapunov function $L(t)=\frac{1}{2}\sum_{j=1}^{r}q_j^2(t)$. 
By using the queueing dynamic equation (\ref{eq:queuedynamic}), it is easy to obtain that:
\begin{eqnarray*}
\frac{1}{2}q_j^2(t+1)-\frac{1}{2}q_j^2(t) \leq \delta_{max}^2 + q_j(t) [A_j(t)-\mu_j(t)].
\end{eqnarray*}
Summing over all $j=1, ..., r$ and adding to both sides the term $Vf(t)$, we obtain:
\begin{eqnarray}
\hspace{-.3in}&& L(t+1)-L(t) + Vf(t)\label{eq:driftcompute_noexp} \\
\hspace{-.3in}&&\qquad\qquad\qquad\leq B^2 + \bigg\{Vf(t) + \sum_{j=1}^rq_j(t) [A_j(t)-\mu_j(t)]\bigg\}. \nonumber
\end{eqnarray}
We see from (\ref{eq:QLAeq}) then given the network state $S(t)$, QLA chooses an action to \emph{minimize the right-hand side (RHS) at time $t$}. Now compare the term in $\{\}$ in the RHS of (\ref{eq:driftcompute_noexp}) with (\ref{eq:dual_separable_si}), we see that we indeed have:
\begin{eqnarray}
L(t+1)-L(t) + Vf^Q(t) &\leq& B^2 + g_{S(t)}(\bv{q}(t)), \label{eq:samplepath_drift_dual}
\end{eqnarray}
where we use $f^Q(t)=f(x^{QLA}(t))$ to denote the utility incurred by QLA's action at time $t$, and $g_{S(t)}(\cdot)$ is the function (\ref{eq:dual_separable_si}) with the network state being $S(t)$.

(Part A: Proof of Utility) We first prove the utility performance. 
Consider $t=0$ and first assume that $S(0)=s_1$. 
Summing up the inequality (\ref{eq:samplepath_drift_dual}) from time $t=0$ to time $t=T_{1}(0)-1$, we have:
\begin{eqnarray*}
L(T_{1}(0))-L(0) + \sum_{t=0}^{T_{1}(0)-1}Vf^Q(t) \leq T_{1}(0)B^2 \qquad\qquad\\+ \sum_{t=0}^{T_{1}(0)-1}g_{S(t)}(\bv{q}(t)). 
\end{eqnarray*}
This can be rewritten as: 
\begin{eqnarray}
\hspace{-.3in}&&L(T_{1}(0))-L(0) + \sum_{t=0}^{T_{1}(0)-1}Vf^Q(t)\leq T_{1}(0)B^2 \label{eq:samplepath_drift_sum}\\
\hspace{-.3in}&& \,\,\,+ \sum_{t=0}^{T_{1}(0)-1}g_{S(t)}(\bv{q}(0)) + \sum_{t=0}^{T_{1}(0)-1}\big[ g_{S(t)}(\bv{q}(t)) - g_{S(t)}(\bv{q}(0))\big].\nonumber
\end{eqnarray}
Using (\ref{eq:dualbddslope_si}) and the fact that $||\bv{q}(t+\tau)-\bv{q}(t)||\leq \tau B$, we see that the final term can be bounded by:
\begin{eqnarray*}
\bigg|\sum_{t=0}^{T_{1}(0)-1}\big[ g_{S(t)}(\bv{q}(t)) - g_{S(t)}(\bv{q}(0))\big] \bigg| \leq \sum_{t=0}^{T_{1}(0)-1} t B^2\qquad\qquad \\
= \big[\frac{1}{2}(T_1(0))^2-\frac{1}{2}T_1(0)\big]B^2.
\end{eqnarray*}
Plugging this into (\ref{eq:samplepath_drift_sum}), and letting $\hat{C}=\frac{1}{2}(T_1(0))^2+\frac{1}{2}T_1(0)$, we obtain:
\begin{eqnarray}
\hspace{-.3in}&&L(T_{1}(0))-L(0) + \sum_{t=0}^{T_{1}(0)-1}Vf^Q(t)\\
\hspace{-.3in}&&\qquad\qquad\qquad\leq \hat{C}B^2 + \sum_{t=0}^{T_{1}(0)-1}g_{S(t)}(\bv{q}(0))\nonumber\\
\hspace{-.3in}&& \qquad\qquad\qquad = \hat{C}B^2 + \sum_{s_i} n^{T_1(0)}_{s_i}(0) g_{s_i}(\bv{q}(0)).\nonumber
\end{eqnarray}
Here $n^{T_1(0)}_{s_i}(t_0)$ denotes the number of times the network state $s_i$ appears in the period $[t_0, t_0+T_1(0)-1]$. Now we take expectations over $T_1(0)$ on both sides conditioning on $S(0)=s_1$ and $\bv{q}(0)$, we have:
\begin{eqnarray}
\hspace{-.3in}&&\expect{L(T_{1}(0))-L(0)\left.|\right. S(0), \bv{q}(0)} \label{eq:visittime-expr}\\
\hspace{-.3in}&& \qquad\qquad\qquad+ \expect{\sum_{t=0}^{T_{1}(0)-1}Vf^Q(t)\left.|\right. S(0), \bv{q}(0)}\nonumber\\
\hspace{-.3in}&& \qquad\qquad \leq CB^2 + \sum_{s_i} \expect{n^{T_1(0)}_{s_i}(0)\left.|\right. S(0), \bv{q}(0)}g_{s_i}(\bv{q}(0)).\nonumber
\end{eqnarray}
Here $C=\expect{\hat{C}\left.|\right. S(0), \bv{q}(0)}=\frac{1}{2}[\overline{T_1^2}+\overline{T_{1}}]$. The above equation uses the fact that $g_{s_i}(\bv{q}(0))$ is a constant given $\bv{q}(0)$. Now by Theorem 2 in Page 154 of \cite{gallager_stochastic} we have that:
\begin{equation}
\expect{n^{T_i(0)}_{s_i}(0)\left.|\right. S(0), \bv{q}(0)}=\frac{\pi_{s_i}}{\pi_{s_1}}. \label{eq:visittime-renewal}
\end{equation}
Plug this into (\ref{eq:visittime-expr}), we have:
\begin{eqnarray}
\hspace{-.3in}&&\expect{L(T_{1}(0))-L(0)\left.|\right. S(0), \bv{q}(0)} \label{eq:visittime-expr-2}\\
\hspace{-.3in}&& \qquad\qquad\qquad+ \expect{\sum_{t=0}^{T_{1}(0)-1}Vf^Q(t)\left.|\right. S(0), \bv{q}(0)}\nonumber\\
\hspace{-.3in}&& \qquad\qquad \leq CB^2 + \frac{1}{\pi_{s_1}}\sum_{s_i}\pi_{s_i}g_{s_i}(\bv{q}(0)).\nonumber
\end{eqnarray}
Now using (\ref{eq:dualfunction_sumofsi}) and (\ref{eq:exp-time-prob}), i.e., $\overline{T_1}=1/\pi_{s_1}$ and $g(\bv{\gamma})=\sum_{s_i}\pi_{s_i}g_{s_i}(\bv{\gamma})$, we obtain:
\begin{eqnarray}
\hspace{-.3in}&&\expect{L(T_{1}(0))-L(0)\left.|\right. S(0), \bv{q}(0)} \label{eq:drift_expected_dual}\\
\hspace{-.3in}&& \qquad\qquad\qquad+ \expect{\sum_{t=0}^{T_{1}(0)-1}Vf^Q(t)\left.|\right. S(0), \bv{q}(0)}\nonumber\\
\hspace{-.3in}&& \qquad\qquad\quad \leq  CB^2 + \overline{T_{1}} g(\bv{q}(0)).\nonumber
\end{eqnarray}
By Corollary \ref{coro:primaldualequal}, we see that $g(\bv{q}(0))\leq Vf^*_{av}$. Thus we conclude that:
\begin{eqnarray}
\hspace{-.3in}&&\expect{L(T_{1}(0))-L(0)\left.|\right. S(0), \bv{q}(0)} \\
\hspace{-.3in}&& \qquad\qquad\quad+ \expect{\sum_{t=0}^{T_{1}(0)-1}Vf^Q(t)\left.|\right. S(0), \bv{q}(0)}\nonumber\\
\hspace{-.3in}&& \qquad\qquad\quad \leq  CB^2 + \overline{T_{1}} Vf^*_{av}.\nonumber
\end{eqnarray}
More generally, if $t_{k}=t_{k-1}+T_{i}(t_{k-1})$ with $t_0=0$ is the $k^{th}$ time after time $0$ when $S(t)=s_1$, we have:
\begin{eqnarray}
\hspace{-.3in}&&\expect{L(t_{k+1})-L(t_{k})\left.|\right. S(t_{k}), \bv{q}(t_{k})} \\
\hspace{-.3in}&& \qquad\qquad\quad+ \expect{\sum_{t=t_{k}}^{t_{k+1}-1}Vf^Q(t)\left.|\right. S(t_{k}), \bv{q}(t_{k})}\nonumber\\
\hspace{-.3in}&& \qquad\qquad\quad \leq CB^2 + \overline{T_{1}} Vf^*_{av},\nonumber
\end{eqnarray}
Now taking expectations over $\bv{q}(t_{k})$ on both sides, we have:
\begin{eqnarray}
\hspace{-.3in}&&\expect{L(t_{k+1})-L(t_k)\left.|\right. S(t_k)} + \expect{\sum_{t=t_k}^{t_{k+1}-1}Vf^Q(t)\left.|\right. S(t_{k})}\nonumber\\
\hspace{-.3in}&& \qquad\qquad\qquad\qquad\qquad\qquad\quad \leq CB^2 + \overline{T_{1}} Vf^*_{av}.\nonumber
\end{eqnarray}
Note that given $S(0)=s_1$, we have the complete information of $S(t_{k})$ for all $k$. Hence the above is the same as: 
\begin{eqnarray}
\hspace{-.3in}&&\expect{L(t_{k+1})-L(t_k)\left.|\right. S(0)}+ \expect{\sum_{t=t_k}^{t_{k+1}-1}Vf^Q(t)\left.|\right. S(0)} \nonumber\\
\hspace{-.3in}&& \qquad\qquad \qquad\qquad \qquad\qquad\quad \leq CB^2 + \overline{T_{1}} Vf^*_{av}. 
\end{eqnarray}
Summing the above from $k=0$ to $K-1$, we get:
\begin{eqnarray}
\hspace{-.3in}&&\expect{L(t_{K})-L(0)\left.|\right. S(0)=s_1} \label{eq:thesum-ineq}\\
\hspace{-.3in}&& \qquad\qquad\quad+ \expect{\sum_{t=0}^{t_{K}-1}Vf^Q(t)\left.|\right. S(0)=s_1}\nonumber\\
\hspace{-.3in}&& \qquad\qquad\quad \leq KCB^2 + K\overline{T_{1}} Vf^*_{av}.\nonumber
\end{eqnarray}
Using the facts that $|f(t)|\leq\delta_{max}$, $\lceil K\overline{T_1}\rceil\leq K\overline{T_1}+1$, $L(0)=0$ and $L(t)\geq0$ for all $t$, we have:
\begin{eqnarray}
\hspace{-.3in}&&\expect{\sum_{t=0}^{\lceil K\overline{T_1}\rceil-1}Vf^Q(t)\left.|\right. S(0)=s_1}\nonumber\\
\hspace{-.3in}&&\qquad \leq KCB^2 + K\overline{T_{1}} Vf^*_{av} + V\delta_{max}\\
\hspace{-.3in}&&\qquad\qquad\qquad+V\delta_{max}\expect{|K\overline{T_1}-t_K|\left.|\right. S(0)=s_1}.\nonumber
\end{eqnarray}
Dividing both sides by $V\lceil K\overline{T_1}\rceil$, we get: 
\begin{eqnarray}
\hspace{-.3in}&&\frac{1}{\lceil K\overline{T_1}\rceil}\expect{\sum_{t=0}^{\lceil K\overline{T_1}\rceil-1}f^Q(t)\left.|\right. S(0)=s_1} \label{eq:expsamplepath_bound}\\
\hspace{-.3in}&& \qquad\quad\leq \frac{CB^2K}{V\lceil K\overline{T_1} \rceil} +  f^*_{av}\frac{K\overline{T_1}V}{\lceil K\overline{T_1} \rceil} +\frac{V\delta_{max}}{\lceil K\overline{T_1} \rceil}\nonumber\\
\hspace{-.3in}&& \qquad\qquad+ \expect{\big|\frac{t_K-K\overline{T_1}}{K}\big|\left.|\right. S(0)=s_1}\frac{K\delta_{max}}{\lceil K\overline{T_1} \rceil}.\nonumber
\end{eqnarray}
Since $t_K=\sum_{k=0}^{K-1}T_i(t_k)$ with $t_0=0$, and each $T_1(t_k)$ is i.i.d. distributed with mean $\overline{T_1}$ and second moment $\overline{T_1^2}<\infty$, we have:
\begin{eqnarray}
\hspace{-.3in}&&\bigg[\expect{\big|\frac{t_K-K\overline{T_1}}{K}\big|\left.|\right. S(0)=s_1}\bigg]^2\label{eq:switchexp_limit}\\
\hspace{-.3in}&&\qquad\qquad\leq \expect{\big|\frac{t_K-K\overline{T_1}}{K}\big|^2\left.|\right. S(0)=s_1}
\leq \frac{\overline{T_1^2}}{K}.\nonumber
\end{eqnarray}
This implies that the term $\expect{\big|\frac{t_K-K\overline{T_1}}{K}\big|\left.|\right. S(0)=s_1}\rightarrow0$ as $K\rightarrow\infty$. 
It is also easy to see that $\lceil K\overline{T_1} \rceil\rightarrow\infty$ and $\frac{K}{\lceil K\overline{T_1} \rceil}\rightarrow\frac{1}{\overline{T_1}}$ as $K\rightarrow\infty$.  
Thus using (\ref{eq:switchexp_limit}) and taking a limsup as $K\rightarrow\infty$ in (\ref{eq:expsamplepath_bound}), we have:  
\begin{eqnarray*}
\hspace{-.3in}&&\limsup_{K\rightarrow\infty}\frac{1}{\lceil K\overline{T_1} \rceil}\expect{\sum_{t=0}^{\lceil K\overline{T_1} \rceil-1}f^Q(t)\left.|\right. S(0)=s_1}\leq \frac{CB^2}{V\overline{T_1}} +  f^*_{av}. 
\end{eqnarray*}

Now consider the case when the starting state is $s_j\neq s_1$. In this case, let $T_{j1}(0)$ be the first time the system enters state $s_1$. Then we see that the above argument can be repeated for the system starting at time $T_{j1}(0)$. The only difference 
is that now the ``initial'' backlog in this case is given by $\bv{q}(T_{j1}(0))$. Specifically, we have from (\ref{eq:thesum-ineq}) that:
\begin{eqnarray}
\hspace{-.3in}&&\expect{L(\hat{t_{K}})-L(T_{j1}(0))\left.|\right. T_{j1}(0), S(0)=s_j} \\
\hspace{-.3in}&& \qquad\qquad\quad+ \expect{\sum_{t=T_{j1}(0)}^{\hat{t_{K}}-1}Vf^Q(t)\left.|\right. T_{j1}(0), S(0)=s_j}\nonumber\\
\hspace{-.3in}&& \qquad\qquad\quad \leq KCB^2 + K\overline{T_{1}} Vf^*_{av}.\nonumber
\end{eqnarray}
Here $\hat{t_K}$ is the $K^{th}$ return time of $S(t)$ to $s_1$ after time $T_{j1}(0)$. We  thus obtain:
\begin{eqnarray*}
\hspace{-.3in}&&V\expect{\sum_{t=T_{j1}(0)}^{\hat{t_K}-1}f^Q(t)\left.|\right. T_{j1}(0), S(0)=s_j}\\
\hspace{-.3in}&&  \leq KCB^2 +  K\overline{T_1}Vf^*_{av} + \expect{L(T_{j1}(0))\left.|\right. T_{j1}(0), S(0)=s_j}.
\end{eqnarray*}
However, since the increment of each queue is no more than $\delta_{max}$ every time slot, we see that $L(T_{j1}(0))\leq [T_{j1}(0)]^2B^2/2$. Also using the fact that $|f(t)|\leq\delta_{max}$ for all $0\leq t\leq T_{j1}(0)$, we have:
\begin{eqnarray*}
\hspace{-.3in}&&V\expect{\sum_{t=0}^{\hat{t_K}-1}f^Q(t)\left.|\right. T_{j1}(0), S(0)=s_j}\\
\hspace{-.3in}&& \qquad  \leq KCB^2 +  K\overline{T_1}Vf^*_{av} + [T_{j1}(0)]^2B^2/2 + T_{j1}(0)V\delta_{max}.
\end{eqnarray*}
Now taking expectations over $T_{j1}(0)$ on both sides,  and using a similar argument as (\ref{eq:expsamplepath_bound}), we get that for every starting state $s_j$, we have:
\begin{eqnarray*}
\hspace{-.3in}&&\limsup_{K\rightarrow\infty}\frac{1}{\lceil\overline{T_{j1}}+K\overline{T_1}\rceil}\expect{\sum_{t=0}^{\lceil\overline{T_{j1}}+K\overline{T_1}\rceil-1}f^Q(t)\left.|\right. S(0)=s_j}\\
\hspace{-.3in}&&\qquad\qquad\qquad\qquad\qquad\qquad\qquad\qquad\qquad\leq \frac{CB^2}{V\overline{T_1}} +  f^*_{av}. 
\end{eqnarray*}
This proves the utility part (\ref{eq:qla_utility}). 



(Part B: Proof of Backlog) Now we look at the backlog performance of QLA. We similarly first assume $S(0)=s_1$. 
Recall that equation (\ref{eq:drift_expected_dual}) says:
\begin{eqnarray}
\hspace{-.3in}&&\expect{L(T_{1}(0))-L(0)\left.|\right. S(0), \bv{q}(0)} \label{eq:drift_expected_dual2} \\
\hspace{-.3in}&& \qquad\qquad\qquad+ \expect{\sum_{t=0}^{T_{1}(0)-1}Vf^Q(t)\left.|\right. S(0), \bv{q}(0)}\nonumber\\
\hspace{-.3in}&& \qquad\quad \leq  CB^2 + \overline{T_1} g(\bv{q}(0)).\nonumber
\end{eqnarray}
Using the definition of $g_c(\bv{\gamma})$ defined in (\ref{eq:dual_convex}) in Appendix A, plugging the set of $\{\vartheta^{(s_i)}_k\}_{i=1, 2, ...}^{k=1, 2, ..., r+2}$ variables and the set of actions $\{x^{(s_i)}_k\}_{i=1, 2, ...}^{k=1,2, ..., r+2}$ in the slackness assumption (\ref{eq:slackness}), and using the facts that $g(\bv{\gamma})=g_c(\bv{\gamma})$ and $0\leq f(t)\leq\delta_{max}$, it can be shown that $g(\bv{\gamma})$ satisfies: 
\begin{eqnarray}
g(\bv{\gamma}) \leq V\delta_{max} - \eta \sum_{j=1}^r \gamma_j.\label{eq:dualproperty}
\end{eqnarray}
Using this in (\ref{eq:drift_expected_dual2}), we have: 
\begin{eqnarray}
\hspace{-.3in}&&\expect{L(T_{1}(0))-L(0)\left.|\right. S(0), \bv{q}(0)} \\
\hspace{-.3in}&& \qquad\qquad\qquad+ \expect{\sum_{t=0}^{T_{1}(0)-1}Vf^Q(t)\left.|\right. S(0), \bv{q}(0)}\nonumber\\
\hspace{-.3in}&& \qquad\quad  \leq CB^2 + \overline{T_1}V\delta_{max} - \overline{T_1}\eta\sum_{j=1}^rq_j(0).\nonumber
\end{eqnarray}
More generally, we have:
\begin{eqnarray}
\hspace{-.3in}&&\expect{L(t_{k+1})-L(t_k)\left.|\right. S(t_{k}), \bv{q}(t_{k})} \label{eq:drift_backlog_epsilon}\\
\hspace{-.3in}&&  \qquad\qquad\qquad\qquad\leq CB^2 + \overline{T_1}V\delta_{max} - \overline{T_1} \eta\sum_{j=1}^rq_j(t_{k}). \nonumber
\end{eqnarray}
Here $t_k$ is the $k^{th}$ return time of $S(t)$ to state $s_1$ after time $0$. 
Taking expectations on both sides over $\bv{q}(t_k)$ and rearranging the terms, we get:
\begin{eqnarray}
\hspace{-.3in}&&\expect{L(t_{k+1})-L(t_k)\left.|\right. S(t_{k})} \label{eq:drift_backlog_epsilon2}\\
\hspace{-.3in}&&\qquad+ \overline{T_1} \eta\sum_{j=1}^r\expect{q_j(t_{k})\left.|\right. S(t_{k})} \leq CB^2 + \overline{T_1}V\delta_{max}. \nonumber
\end{eqnarray}
Now using the fact that conditioning on $S(t_k)$ is the same as conditioning on $S(0)$, we have:
\begin{eqnarray}
\hspace{-.3in}&&\expect{L(t_{k+1})-L(t_k)\left.|\right. S(0)} \label{eq:drift_backlog_epsilon3}\\
\hspace{-.3in}&&\qquad+ \overline{T_1} \eta\sum_{j=1}^r\expect{q_j(t_{k})\left.|\right. S(0)} \leq CB^2 + \overline{T_1}V\delta_{max}. \nonumber
\end{eqnarray}
Summing over $k=0, ..., K-1$, rearranging the terms, and using the facts that $L(0)=0$ and $L(t)\geq0$ for all $t$:  
\begin{eqnarray}
\hspace{-.5in}&&\sum_{k=0}^{K-1}\overline{T_1} \eta\sum_{j=1}^r\expect{q_j(t_{k})\left.|\right. S(0)} \leq KCB^2 + K\overline{T_1}V\delta_{max}. \label{eq:drift_backlog_epsilon4}
\end{eqnarray}
Dividing both sides by $K\overline{T_1}\eta$, we get:
\begin{eqnarray}
\hspace{-.3in}&&\frac{1}{K}\sum_{k=0}^{K-1}\sum_{j=1}^r\expect{q_j(t_{k})\left.|\right. S(0)} \leq \frac{CB^2 + \overline{T_1}V\delta_{max} }{\overline{T_1} \eta}. \label{eq:drift_backlog_epsilon5}
\end{eqnarray}
Now using the fact that $|q_j(t+\tau)-q_j(t)|\leq\tau\delta_{max}$, we have:
\begin{eqnarray*}
\sum_{\tau=t_{k}}^{t_{k+1}-1}\sum_{j=1}^rq_j(\tau)\leq T_{1}(t_k)\sum_{j=1}^rq_j(t_{k})\qquad\qquad\\
+[\frac{1}{2}(T_1(t_k))^2-\frac{1}{2}T_1(t_k)]B^2.
\end{eqnarray*}
Taking expectations on both sides conditioning on $S(0)$ (which is the same as conditioning on $S(t_k)$), we get:
\begin{eqnarray*}
\hspace{-.3in}&&\expect{\sum_{\tau=t_{k}}^{t_{k+1}-1}\sum_{j=1}^rq_j(\tau)\left.|\right. S(0)}\\
\hspace{-.3in}&&\qquad\leq \expect{T_{1}(t_k)\sum_{j=1}^rq_j(t_{k})\left.|\right. S(0)}
+[\frac{1}{2}\overline{T_1^2}-\frac{1}{2}\overline{T_1}]B^2\\
\hspace{-.3in}&&\qquad = \overline{T_1}\expect{\sum_{j=1}^rq_j(t_{k})\left.|\right. S(0)}
+[\frac{1}{2}\overline{T_1^2}-\frac{1}{2}\overline{T_1}]B^2.
\end{eqnarray*}
In the last step, we have used the fact that $T_1(t_k)$ is independent of $\bv{q}(t_k)$. 
Summing the above equation over $k=0, 1, ..., K-1$, we have:
\begin{eqnarray*}
\hspace{-.3in}&&\expect{\sum_{t=0}^{t_K-1}\sum_{j=1}^rq_j(t)\left.|\right. S(0)} \\
\hspace{-.3in}&&\qquad\leq \sum_{k=0}^{K-1} \overline{T_1}\expect{\sum_{j=1}^rq_j(t_k)\left.|\right. S(0)} + \frac{K[\overline{T_1^2}-\overline{T_1}]B^2}{2}.
\end{eqnarray*}
Dividing both sides by $K$ and using (\ref{eq:drift_backlog_epsilon5}), we have:
\begin{eqnarray*}
\hspace{-.3in}&&\frac{1}{K}\expect{\sum_{t=0}^{t_K-1}\sum_{j=1}^rq_j(t)\left.|\right. S(0)} \\
\hspace{-.3in}&&\leq\frac{\overline{T_1}}{K} \sum_{k=0}^{K-1} \expect{\sum_{j=1}^rq_j(t_k)\left.|\right. S(0)} + \frac{[\overline{T_1^2}-\overline{T_1}]B^2}{2}\\
\hspace{-.3in}&&\leq \frac{CB^2 + \overline{T_1}V\delta_{max} }{\eta} + \frac{[\overline{T_1^2}-\overline{T_1}]B^2}{2}. 
\end{eqnarray*}
Now notice that we always have $t_K\geq K$. Hence:
\begin{eqnarray*}
\hspace{-.3in}&&\frac{1}{K}\sum_{t=0}^{K-1}\expect{\sum_{j=1}^rq_j(t)\left.|\right. S(0)} \leq\frac{1}{K}\expect{\sum_{t=0}^{t_K-1}\sum_{j=1}^rq_j(t)\left.|\right. S(0)} \\
\hspace{-.3in}&&\qquad\qquad\qquad\qquad \leq \frac{CB^2 + \overline{T_1}V\delta_{max} }{\eta} +\frac{[\overline{T_1^2}-\overline{T_1}]B^2}{2}. 
\end{eqnarray*}
This proves  (\ref{eq:qla_backlog}) for the case when $S(0)=s_1$. The case when $S(0)=s_j\neq s_1$ can be treated in a similar way as in Part A. It can be shown that the above backlog bound still holds, as the effect of the backlog values before the first hitting time $T_{j1}(0)$ will vanish as time increases. 
This proves the backlog bound (\ref{eq:qla_backlog}).
Theorem \ref{thm:generalq} thus follows by combining the two proofs. 
\end{proof}


%

\section*{Appendix A- Proof of Theorem \ref{thm:primaldualequal}}
We now prove Theorem \ref{thm:primaldualequal}. The proof idea is shown in Fig. \ref{figure:primaldual}, and can be described as follows: First we construct a ``convexified'' version of the deterministic problem (\ref{eq:primal}) and show that it gives the exact value of $Vf^*_{av}$. 
We then show that the dual function $g_c(\bv{\gamma})$ of this convexified problem,  is exactly the same as the dual function $g(\bv{\gamma})$ of (\ref{eq:dualproblem}). Hence the two dual problems have the same optimal value. We finally show that the duality gap is zero for the convexified problem by showing that its ``utility-constraint'' set is convex. Hence $g(\bv{\gamma}^*_V)=g^*_c=Vf^*_{av}$, where $g^*_c$ is the optimal value of the dual problem for the convexified problem. 
\begin{figure}[cht]
\centering
\includegraphics[height=1.5in, width=3.5in]{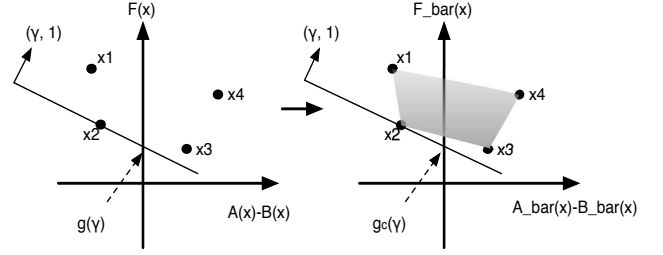}
\caption{The left figure shows the utility-constraint set of the deterministic problem with $r=1, M=1$ and its dual function. The right figure shows the utility-constraint set of the corresponding ``convexified'' problem and its dual function. It can be seen that the two dual functions are the same, and that the ``convexified'' problem has no duality gap.  }\label{figure:primaldual}
\end{figure}

\begin{proof} (Theorem \ref{thm:primaldualequal})
For notation simplicity, we denote the set of $\bv{x}=(x^{(s_1)}, x^{(s_2)}, ...)$, $x^{(s_i)}\in\script{X}^{(s_i)}$ as $\bv{\script{X}}$. We then consider the following modified deterministic problem:
\begin{eqnarray}
\hspace{-.4in}&&\min:\,\,\, \overline{\script{F}}(\{a_k^{(s_i)}, \bv{x}^k\})\triangleq V\sum_{s_i}\pi_{s_i}\sum_{k=1}^{r+2}a^{(s_i)}_kf(s_i, x^{(s_i)}_k)\label{eq:primal_convex}\\
\hspace{-.4in}&&s.t.\quad\,\,\overline{\script{A}}_j(\{a_k^{(s_i)}, \bv{x}^k\})\triangleq\sum_{s_i}\pi_{s_i}\sum_{k=1}^{r+2}a^{(s_i)}_kA_j(s_i, x^{(s_i)}_k)\nonumber\\
\hspace{-.4in}&&\qquad\qquad\leq \overline{\script{B}}_j(\{a_k^{(s_i)}, \bv{x}^k\})\triangleq\sum_{s_i}\pi_{s_i}\sum_{k=1}^{r+2}a^{(s_i)}_k\mu_j(s_i, x^{(s_i)}_k),\nonumber\\
\hspace{-.4in}&&\qquad\quad \bv{x}^k\in \bv{\script{X}}\quad \forall\, k=1, ..., r+2,\nonumber\\
\hspace{-.4in}&&\qquad\quad a^{(s_i)}_k\geq0, \sum_{k=1}^{r+2}a^{(s_i)}_k=1, \,\,\forall\,\, s_i.\nonumber
\end{eqnarray}
Here $\bv{x}^k=(x^{(s_i)}_k, x^{(s_2)}_k, ...)$. 
Due to the use of the auxiliary variables $\{a^{(s_i)}_k\}$, this problem can be viewed as the ``convexified'' version of the original deterministic problem (\ref{eq:primal}). Denote the optimal value of (\ref{eq:primal_convex}) as $OPT_c$. \footnote{Without loss of generality, we assume such an optimal value exists. Else we can replace the ``min'' with ``inf'' in (\ref{eq:primal_convex}), 
consider an $\epsilon$-optimal solution and let $\epsilon\rightarrow0$. Below we will use similar assumptions about the existence of an optimal policy for the stochastic problem, and the existence of an optimal solution to (\ref{eq:primal_convex}). } We will prove Theorem \ref{thm:primaldualequal} via the following two claims. The first claim shows that $OPT_c=Vf^*_{av}$ and the second claim shows that $OPT_c=g(\bv{\gamma}^*_V)$. 

\begin{claim}\label{claim:primal_opt}
$Vf^*_{av}=OPT_c$
\end{claim}
\begin{proof} (Claim \ref{claim:primal_opt}): 
For each action vector $\bv{x}\in\bv{\script{X}}$, we define its ``utility-constraint'' vector $\bv{J}(\bv{x})$ as follows: 
\begin{eqnarray*}
\bv{J}(\bv{x})=(\script{F}(\bv{x}), \script{A}_1(\bv{x})-\script{B}_1(\bv{x}), ..., \script{A}_r(\bv{x})-\script{B}_r(\bv{x})). 
\end{eqnarray*}
Denote $\script{J}=\{\bv{J}(\bv{x}): \bv{x}\in\bv{\script{X}}\}$, i.e., $\script{J}$ is the set of all possible utility-constraint vectors for $\bv{\script{X}}$, and denote $\overline{\script{J}}$ the convex hull of $\script{J}$. Let $\Pi^*$ be an optimal action-choosing policy that solves the stochastic problem.  Now define the ``utility-constraint'' vector $\bv{J}$ for $\Pi^*$ as: 
\begin{eqnarray*}
\bv{J}^{\Pi^*}=(\overline{f^{\Pi^*}}, \overline{A^{\Pi^*}_1}-\overline{\mu^{\Pi^*}_1}, ..., \overline{A^{\Pi^*}_r}-\overline{\mu^{\Pi^*}_r}), 
\end{eqnarray*}
where $\overline{f^{\Pi^*}}=f^*_{av}$ is the time average cost under $\Pi^*$, and $\overline{A^{\Pi^*}_j}, \overline{\mu^{\Pi^*}_j}$ are the time average input and output rates to queue $j$ under $\Pi^*$. Note that here we have assumed without loss of generality that the time averages converge. \footnote{ In the case when this assumption is violated, the same argument can be applied to the limit points of the time averages but is more involved.}
It can then be shown by using an argument similar to that in \cite{neelyenergy} that the vector $\bv{J}^{\Pi^*}\in\overline{\script{J}}$. Using Caratheodory's theorem \cite{bertsekasoptbook}, we see then there exist $\{a^{(s_i)}_k\}_{i=1, 2, ...}^{k=1,..., r+2}$ with $a^{(s_i)}_k\geq0$, $\sum_ka^{(s_i)}_k=1$, and a set of action vectors $\{\bv{x}^k\}_{k=1}^{r+2}\subset\bv{\script{X}}$ such that:
\begin{eqnarray*}
\hspace{-.3in}&& \sum_{s_i}\pi_{s_i}\sum_{k=1}^{r+2}a^{(s_i)}_kf(s_i, x^{(s_i)}_k) 
 = \overline{f^{\Pi^*}},\\
\hspace{-.3in}&&\sum_{s_i}\pi_{s_i}\sum_{k=1}^{r+2}a^{(s_i)}_kA_j(s_i, x^{(s_i)}_k)-\sum_{s_i}\pi_{s_i}\sum_{k=1}^{r+2}a^{(s_i)}_k\mu_j(s_i, x^{(s_i)}_k)\\
 \hspace{-.3in}&&\qquad\qquad\qquad\qquad\qquad\qquad=\overline{A^{\Pi^*}_j}-\overline{\mu^{\Pi^*}_j}\leq 0 \quad\forall\,\, j. 
\end{eqnarray*}
The inequality $\overline{A^{\Pi^*}_j}-\overline{\mu^{\Pi^*}_j}\leq 0$ holds since $\Pi^*$ is by definition a stabilizing policy. This shows that $\{a^{(s_i)}_k\}_{i=1, 2, ...}^{k=1,..., r+2}, \{\bv{x}^k\}_{k=1}^{r+2}$ is a feasible solution of (\ref{eq:primal_convex}), implying $Vf_{av}^*\geq OPT_c$.

To prove the other direction, let $\{\overline{a}^{(s_i)}_k\}^{k=1, ..., r+2}_{i=1, 2, ...}$ and $\{\overline{\bv{x}}^{k}\}_{k=1}^{r+2}$ be an optimal solution pair of (\ref{eq:primal_convex}). Now by our slackness assumption (\ref{eq:slackness}), there exists a set of actions $\{x^{(s_i)}_k\}_{i=1, 2, ...}^{k=1, ..., r+2}$ and probabilities $\{\vartheta^{(s_i)}_k\}_{i=1, 2, ...}^{k=1, ..., r+2}$ with $\sum_k\vartheta^{(s_i)}_k=1$ such that $\sum_{s_i}\pi_{s_i}\big\{\sum_k\vartheta^{(s_i)}_k[A_{j}(s_i, x^{(s_i)}_k)-\mu_{j}(s_i, x^{(s_i)}_k)]\big\}\leq -\eta$ for some $\eta>0$ for all $j$. We can thus construct the following policy $\Pi'$: fix some $\hat{\epsilon}\in(0, 1)$, at every state $s_i$, choose action $\overline{x}^{(s_i)}_k$ with probability $(1-\hat{\epsilon})a^{(s_i)}_k$ and choose action $x^{(s_i)}_k$ with probability $\hat{\epsilon}\vartheta^{(s_i)}_k$. 
Since $|f(t)|\leq\delta_{max}$ for all $t$, it is easy to see then under $\Pi'$: 
\begin{eqnarray}
|\overline{f^{\Pi'}} - OPT_c/V|\leq\hat{\epsilon}\delta_{max}, \label{eq:OPTc-opt}
\end{eqnarray}
and that for each queue $j$, $\overline{A^{\Pi'}_j}-\overline{\mu^{\Pi'}_j}<-\hat{\epsilon}\eta<0$. This policy can be shown to ensure that the network is strongly stable. Hence $\Pi'$ is a feasible control policy. Therefore $\overline{f^{\Pi'}}\geq f_{av}^*$ by the definition of $f_{av}^*$. Using this fact together with (\ref{eq:OPTc-opt}), we have $OPT_c/V\geq f_{av}^*-\hat{\epsilon}\delta_{max}$. Since this holds for all $\hat{\epsilon}>0$, we have $OPT_c/V\geq f_{av}^*$.
\end{proof}

\begin{claim}\label{claim:nodualgap}
$OPT_c=g(\bv{\gamma}^*_V)$
\end{claim}
\begin{proof}
We first look at the dual problem of (\ref{eq:primal_convex}):
\begin{eqnarray}
\max: \,\,\,g_c(\bv{\gamma}),\quad s.t. \,\,\, \bv{\gamma}\succeq\bv{0},\label{eq:dualproblem_convex}
\end{eqnarray}
where $g_c(\bv{\gamma})$ is defined:
\begin{eqnarray}
\hspace{-.3in}&&g_c(\bv{\gamma})=\inf_{x^{(s_i)k}\in \script{X}^{(s_i)}, a^{(s_i)}_k}\sum_{s_i}\pi_{s_i}\bigg\{ V\sum_{k=1}^{r+2}a^{(s_i)}_kf(s_i, x^{(s_i)}_k)\label{eq:dual_convex}\\
\hspace{-.3in}&&\quad+\sum_j\gamma_j\bigg[ \sum_{k=1}^{r+2}a^{(s_i)}_kA_j(s_i, x^{(s_i)}_k)- \sum_{k=1}^{r+2}a^{(s_i)}_k\mu_j(s_i, x^{(s_i)}_k)\bigg]\bigg\}.\nonumber
\end{eqnarray}
Now by comparing (\ref{eq:dual_convex}) and (\ref{eq:dual_separable}), we see that $g_c(\bv{\gamma})=g(\bv{\gamma})$ for all $\bv{\gamma}\succeq\bv{0}$. 
This is so because  at any $\bv{\gamma}\succeq\bv{0}$, we first have $g_c(\bv{\gamma})\leq g(\bv{\gamma})$. Now if $\{x^{(s_i)}\}_{i=1}^{\infty}$ are the minimizers of $g(\bv{\gamma})$, then $\{x^{(s_i)}_k\}_{i=1, 2, ...}^{k=1, ..., r+2}$, with $x^{(s_i)}_k=x^{(s_i)}$, $a^{(s_i)}_1=1$ and $a^{(s_i)}_k=0$ if $k\neq1$, will also be the minimizers of $g_c(\bv{\gamma})$. This shows $g_c(\bv{\gamma}) = g(\bv{\gamma})$, which then implies that $g^*_c=g(\bv{\gamma}^*_V)$, where $g^*_c$ is the optimal value of (\ref{eq:dualproblem_convex}). 

Now it remains to show that $g_c^*=OPT_c$. It suffices to show that $g_c^*\geq OPT_c$. We prove this claim by using a similar approach as that in Page 234 of  \cite{boydconvexopt}. Denote the set $\Gamma=\{\{a^{(s_i)}_k\}_{i=1, 2, ...}^{k=1, ..., r+2}: a^{(s_i)}_k\geq0, \sum_ka^{(s_i)}_k=1,\,\forall\, s_i\}$. Consider the set $\script{M}$ as follows:
\begin{eqnarray*}
\script{M}=\bigg\{(u, c_1 ..., c_r)\left.|\right. \exists\,\, \{a^{(s_i)}_k\} \in\Gamma, \{\bv{x}^{k}\}_{k=1}^{r+2 }\subset\bv{\script{X}}\,\, s.t.\, \qquad\\
\overline{\script{F}}(\{a^{(s_i)}_k, \bv{x}^k\})\leq u, \text{and} \qquad\qquad\qquad\qquad\quad\\
\overline{\script{A}}_j(\{a^{(s_i)}_k, \bv{x}^k\})-\overline{\script{B}}_j(\{a^{(s_i)}_k, \bv{x}^k\})\leq c_j, \,\,\forall j\bigg\}. 
\end{eqnarray*}
It is not difficult to show that $\overline{\script{J}}\subset\script{M}$. We now show that $\script{M}$ is convex. Indeed, if two vectors $(u, c_1 ..., c_r)$ and $(\hat{u}, \hat{c}_1 ..., \hat{c}_r)$ are both in $\script{M}$, then there exist $\{a^{(s_i)}_k\}_{i=1, 2, ...}^{k=1, ..., r+2}, \{\bv{x}^{k}\}_{k=1}^{r+2 }$ and $\{\hat{a}^{(s_i)}_k\}_{i=1, 2, ...}^{k=1, ..., r+2}, \{\hat{\bv{x}}^{k}\}_{k=1}^{r+2 }$ such that:
\begin{eqnarray*}
&&\overline{\script{F}}(\{a^{(s_i)}_k, \bv{x}^k\})\leq u,\quad\,\,\overline{\script{F}}(\{\hat{a}^{(s_i)}_k, \hat{\bv{x}}^k\})\leq \hat{u},\\
&&\overline{\script{A}}_j(\{a^{(s_i)}_k, \bv{x}^k\})- \overline{\script{B}}_j(\{a^{(s_i)}_k, \bv{x}^k\})\leq c_j, \,\,\forall j,\\
&&\overline{\script{A}}_j(\{\hat{a}^{(s_i)}_k, \hat{\bv{x}}^k\})- \overline{\script{B}}_j(\{\hat{a}^{(s_i)}_k, \hat{\bv{x}}^k\})\leq \hat{c}_j, \,\,\forall j.
\end{eqnarray*}
Now if we consider the vectors $\theta\cdot(u, c_1, ..., c_r)+(1-\theta)\cdot (\hat{u}, \hat{c}_1 ..., \hat{c}_r)$. Using Caratheodory's theorem again, we see that there exists $\{\tilde{a}^{(s_i)}_k\}^{k=1, ...,r+2}_{i=1, 2, ...}, \{\tilde{\bv{x}}^{k}\}_{k=1}^{r+2 }$ such that:
\begin{eqnarray*}
\hspace{-.3in}&&\overline{\script{F}}(\{\tilde{a}^{(s_i)}_k, \tilde{\bv{x}}^{k}\})=\theta\overline{\script{F}}(\{a^{(s_i)}_k, \bv{x}^k\}) + (1-\theta)\overline{\script{F}}(\{\hat{a}^{(s_i)}_k, \hat{\bv{x}}^k\}),  \\
\hspace{-.3in}&&\overline{\script{A}}_j(\{\tilde{a}^{(s_i)}_k, \tilde{\bv{x}}^k\})- \overline{\script{B}}_j(\{\tilde{a}^{(s_i)}_k, \tilde{\bv{x}}^k\})\\
\hspace{-.3in}&&\qquad\qquad\quad=\theta\bigg[\overline{\script{A}}_j(\{a^{(s_i)}_k, \bv{x}^k\})- \overline{\script{B}}_j(\{a^{(s_i)}_k, \bv{x}^k\})\bigg] \\
\hspace{-.3in}&&\qquad\qquad\qquad+(1-\theta)\bigg[\overline{\script{A}}_j(\{\hat{a}^{(s_i)}_k, \hat{\bv{x}}^k\})- \overline{\script{B}}_j(\{\hat{a}^{(s_i)}_k, \hat{\bv{x}}^k\})\bigg].
\end{eqnarray*}
This implies that $\theta\cdot(u, c_1, ..., c_r)+(1-\theta)\cdot (\hat{u}, \hat{c}_1 ..., \hat{c}_r)\in\script{M}$, hence $\script{M}$ is convex. 

We now define a second convex set $\script{D}$ as $\script{D}=\{(\nu, c_1, ...,c_r)\left.|\right. \nu<OPT_c, c_j=0, \,\forall\,j \}$. 
It is easy to see then $\script{M}\cap\script{D}$ is empty, for otherwise $OPT_c$ can not be the optimal value of (\ref{eq:dualproblem_convex}). Therefore there exists a hyperplane with norm $(\zeta, \gamma_1, ..., \gamma_r)\neq\bv{0}$ and some constant $c$ such that:
\begin{eqnarray}
(u, c_1, ..., c_r)\in\script{M} &\Rightarrow& u\zeta+\sum_{k=1}^r\gamma_jc_j\geq c, \nonumber\\
(\nu, c_1, ..., c_r)\in\script{D} &\Rightarrow& \nu\zeta+\sum_{k=1}^r\gamma_jc_j\leq c. \label{eq:epigrapheq}
\end{eqnarray}
We can thus conclude that $\zeta\geq0, \gamma_j\geq0$ and $\zeta \nu\leq c$ for all $\nu<OPT_c$, which implies $\zeta OPT_c\leq c$. Using these in (\ref{eq:epigrapheq}), and using the fact that $\overline{\script{J}}\subset\script{M}$, we see that for any $\{a^{(s_i)}_k\}_{i=1, 2, ...}^{k=1, ..., r+2}\in\Gamma, \{\bv{x}^{k}\}_{k=1}^{r+2 }\subset\bv{\script{X}}$, we have:
\begin{eqnarray}
\hspace{-.3in}&&\zeta OPT_c\leq c\leq\zeta\overline{\script{F}}(\{a^{(s_i)}_k, \bv{x}^k\})\label{eq:epicgraph_primaldual} \\ 
\hspace{-.3in}&&\qquad\qquad\qquad+\sum_{j=1}^{r}\gamma_j\bigg[\overline{\script{A}}_j(\{a^{(s_i)}_k, \bv{x}^k\})-\overline{\script{B}}_j(\{a^{(s_i)}_k, \bv{x}^k\})\bigg]. \nonumber
\end{eqnarray}
Clearly, $\zeta\neq0$, for otherwise we can plug in the actions $\{x^{(s_i)}_k\}_{i=1, 2, ...}^{k=1, ..., r+2}$ and probabilities $\{\vartheta^{(s_i)}_k\}_{i=1, 2, ...}^{k=1, ..., r+2}$ in the slackness assumption (\ref{eq:slackness}) to obtain:
\begin{eqnarray*}
0\geq\sum_{j=1}^r\gamma_j(-\eta)\geq0,
\end{eqnarray*}
which will imply $(\zeta, \gamma_1, ..., \gamma_r)=\bv{0}$. Thus we see that $\zeta>0$. Now dividing $\zeta$ from both sides of (\ref{eq:epicgraph_primaldual}), we have:
\begin{eqnarray*}
\overline{\script{F}}(\{a^{(s_i)}_k, \bv{x}^k\})+\sum_{j=1}^{r}\gamma'_j\bigg[\overline{\script{A}}_j(\{a^{(s_i)}_k, \bv{x}^k\})-\overline{\script{B}}_j(\{a^{(s_i)}_k, \bv{x}^k\})\bigg]\\
 \geq  OPT_c, 
\end{eqnarray*}
where $\gamma'_j=\gamma_j/\zeta$. This implies $g_c(\bv{\gamma}')\geq OPT_c$ with $\bv{\gamma}'=(\gamma'_1, ..., \gamma'_r)^T$. Hence $g_c^*\geq g_c(\bv{\gamma}')\geq OPT_c$, which by weak duality implies $g_c^*= OPT_c$, and so $g(\bv{\gamma}^*_V)=OPT_c$. 
\end{proof}
Combining Claim \ref{claim:primal_opt} and \ref{claim:nodualgap}, we see that Theorem \ref{thm:primaldualequal} follows.
\end{proof}

$\vspace{-.2in}$
\bibliographystyle{unsrt}
\bibliography{../mybib}

\end{document}